\documentclass[12pt,reqno]{amsart}

\newcommand\version{January 30, 2011}

\usepackage{amsmath,amsfonts,amsthm,amssymb,amsxtra}

\newtheorem{theorem}{Theorem}
\newtheorem{proposition}[theorem]{Proposition}
\newtheorem{lemma}[theorem]{Lemma}
\newtheorem{corollary}[theorem]{Corollary}

\theoremstyle{definition}

\theoremstyle{remark}

\newtheorem{remark}[theorem]{Remark}

\renewcommand{\epsilon}{\varepsilon}

\newcommand{\N}{\mathbb{N}}

\newcommand{\R}{\mathbb{R}}

\DeclareMathOperator{\Tr}{Tr}
\DeclareMathOperator{\tr}{Tr}

\begin{document}

\title[Eigenvalue bounds --- \version]{Eigenvalue bounds for Schr\"odinger operators\\ with a homogeneous magnetic field}

\author{Rupert L. Frank}
\address{Rupert L. Frank, Department of Mathematics,
Princeton University, Princeton, NJ 08544, USA}
\email{rlfrank@math.princeton.edu}

\author{Rikard Olofsson}
\address{Rikard Olofsson, Department of Mathematics, Uppsala University, Box 480, 75106 Uppsala, Sweden}
\email{rikard.olofsson@math.uu.se}

\thanks{\copyright\, 2011 by the authors. This paper may be reproduced, in its entirety, for non-commercial purposes.\\
The second author acknowledges support through the Swedish Research Council.}

\begin{abstract}
 We prove Lieb-Thirring inequalities for Schr\"odinger operators with a homogeneous magnetic field in two and three space dimensions. The inequalities bound sums of eigenvalues by a semi-classical approximation which depends on the strength of the magnetic field, and hence quantifies the diamagnetic behavior of the system. For a harmonic oscillator in a homogenous magnetic field, we obtain the sharp constants in the inequalities.
\end{abstract}

\maketitle


\section{Introduction and main result}

Lieb-Thirring inequalities \cite{LiTh} provide bounds on the sum of negative eigenvalues of Schr\"odinger operators in terms of a phase space integral. In this paper, we are interested in two-dimensional Schr\"odinger operators $H_B+V$ with a homogenous magnetic field of strength $B>0$. Here
$$
H_B=\left(-i\frac\partial{\partial x_1} + \frac{Bx_2}2 \right)^2 + \left(-i\frac\partial{\partial x_2} - \frac{Bx_1}2 \right)^2
$$
is the Landau Hamiltonian in $L^2(\R^2)$ and $V$ is a real-valued function. The Lieb-Thirring inequality states that
\begin{align}\label{eq:mainnonmag}
\Tr\left( H_{B} + V \right)_- 
\leq r_2\ (2\pi)^{-2}  \iint_{\R^2\times\R^2} \left(|p|^2 + V(x) \right)_- \,dx\,dp
\end{align}
with the (currently best, but presumably non-optimal) constant $r_2=\pi/\sqrt{3}$ from \cite{DLL}. Physically, the left side is (minus) the energy of a system of non-interacting fermions in an external potential $V$ and an external, homogeneous magnetic field of strength $B$, whereas the right side is $-r_2$ times a semi-classical approximation to that energy.

Physically, one expects the system to show a diamagnetic behavior, that is, to have a higher energy in the presence of a magnetic field. This is however not reflected in \eqref{eq:mainnonmag}, which has a right hand side independent of $B$.  We refer to \cite{Fr} for further references and a survey over this problem. Our goal in this paper is to obtain a bound similar to \eqref{eq:mainnonmag}, but with a more refined semi-classical approximation which takes $B$ into account. The approximation we propose is 
\begin{equation}\label{eq:magsymb}
\frac{B}{2\pi} \sum_{m=0}^\infty \int_{\R^2} \left((2m+1)B + V(x) \right)_- \,dx \,.
\end{equation}
This quantity reflects the diamagnetic behavior since
\begin{equation}\label{eq:diamag}
\frac{B}{2\pi} \sum_{m=0}^\infty \int_{\R^2} \left((2m+1)B + V(x)\right)_- \,dx
\leq \frac{1}{(2\pi)^2} \iint_{\R^2\times\R^2} (|p|^2+V(x))_- \,dx\,dp
\end{equation}
for every $V$. Inequality \eqref{eq:diamag} follows (even before the $x$-integration) from an easy convexity inequality (see Lemma \ref{mean} below). We also note that when $B\to 0$, by a Riemann sum argument, the quantity \eqref{eq:magsymb} approaches
\begin{equation}\label{eq:b0}
(4\pi)^{-1} \int_0^\infty dE \int_{\R^2} \left( E+V(x)\right)_- \,dx
= (2\pi)^{-2}  \iint_{\R^2\times\R^2} \left(|p|^2 + V(x)\right)_- \,dx\,dp \,,
\end{equation}
which is the `usual' phase space integral.

While the right side of \eqref{eq:mainnonmag} (up to the constant $r_2$) has the correct limiting behavior when a small parameter $\hbar$ is introduced, it is not useful in the coupled limit $B\to\infty$ and $\hbar\to 0$. This limit is physically relevant, for instance, in the study of neutron stars \cite{LiSoYn}. The magnetic quantity \eqref{eq:magsymb} reproduces the correct behavior in this regime. It is remarkable that this asymptotic profile is, indeed, a uniform, non-asymptotic bound. This is implicitly contained in \cite{LiSoYn2} who use, however, only an approximation of \eqref{eq:magsymb}. Our first result is

\begin{theorem}\label{mainpot}
For any $B>0$ and any $V$ on $\R^2$ one has
\begin{equation}
 \label{eq:mainpot}
\Tr\left( H_B + V \right)_- \leq
\rho_2\ \frac{B}{2\pi} \sum_{m=0}^\infty \int_{\R^2} \left((2m+1)B + V(x) \right)_- \,dx
\end{equation}
with $\rho_2=3$.
\end{theorem}

Hence, up to the moderate increase from $r_2=\pi/\sqrt3\approx 1.81$ to $\rho_2=3$, we have found a magnetic analogue of \eqref{eq:mainnonmag} which reflects the desired diamagnetic behavior \eqref{eq:diamag}. An important ingredient in our proof is a method developed recently by Rumin \cite{Ru} to derive kinetic energy inequalities; see Subsection \ref{sec:kinen}.

\bigskip

Similarly as in the non-magnetic case, one might ask for the optimal value of the constant $\rho_2$. By the semi-classical result mentioned above one necessarily has $\rho_2\geq 1$. A first result in this direction was obtained in \cite{FrLoWe} (extending previous work of \cite{ErLoVo}), where it was shown that if one takes $V$ to be constant on a set of finite measure and plus infinity otherwise, then \eqref{eq:mainpot} holds with $\rho_2=1$. Our second main result is an analogous optimal bound for a harmonic oscillator.

\begin{theorem}\label{main}
For any $B>0$, $\omega_1>0$, $\omega_2>0$ and $\mu>0$, inequality \eqref{eq:mainpot} holds with $\rho_2=1$ for $V(x) = \omega_1^2x_1^2+\omega_2^2x_2^2-\mu$.
\end{theorem}

In particular, letting $B\to 0$ and using the limit in \eqref{eq:b0} we recover the known bounds in the non-magnetic case from \cite{dB,La2}. Even though the eigenvalues of a harmonic oscillator in a homogeneous magnetic field are explicitly known (Lemma \ref{explicit}), the proof of Theorem \ref{main} relies on a delicate property of a subclass of convex functions (Lemma~\ref{mono}) which, we feel, could be useful even beyond the context of this paper.

\subsection*{Moments of eigenvalues}

Using some by now standard techniques we derive a few consequences of Theorems \ref{mainpot} and \ref{main}. First, following Aizenman and Lieb \cite{AiLi} one can replace $V$ by $V-\mu$ in \eqref{eq:mainpot} and integrate with respect to $\mu$ to obtain that for any $\gamma\geq 1$
\begin{equation}\label{eq:mainpotal}
\Tr\left( H_B + V \right)_-^\gamma \leq
\rho_2\ \frac{B}{2\pi} \sum_{m=0}^\infty \int_{\R^2} \left((2m+1)B + V(x) \right)_-^\gamma \,dx\,,
\end{equation}
where $\rho_2=3$ for general $V$ and $\rho_2=1$ for $V(x) = \omega_1^2x_1^2+\omega_2^2x_2^2-\mu$. The restriction $\gamma\geq 1$ is necessary, since one easily checks that for $0\leq\gamma<1$ there is no constant $\rho_2$ such that \eqref{eq:mainpotal} holds for all potentials $V$. Restricting ourselves to the quadratic case we shall show in Subsection \ref{sec:gamma-}

\begin{proposition}
\label{sharpness}
For any $0\leq\gamma<1$ there are $B>0$, $\mu>0$ and $\omega_1=\omega_2$ such that for $V(x)=\omega_1^2x_1^2+\omega_2^2x_2^2-\mu$ one has
\begin{equation}
 \label{eq:gamma-}
\Tr\left( H_B + V \right)_-^\gamma >
\frac{B}{2\pi} \sum_{m=0}^\infty \int_{\R^2} \left((2m+1)B + V(x) \right)_-^\gamma \,dx \,.
\end{equation}
\end{proposition}

Our counterexample appears in the limit $\omega_j/B\to 0$ (with $\mu/B=3$ fixed).

\subsection*{Three dimensions}

Next, we shall show that our bounds for $d=2$ can be applied to deduce analogous bounds for $d=3$. This argument is in the spirit of the lifting argument from \cite{La1,La2,LaWe}. We denote by $\hat H_{B} = H_{B} - \frac{\partial^2}{\partial x_3^2}$ the Landau Hamiltonian in $L^2(\R^3)$.

\begin{corollary}\label{mainpot3d}
For any $B>0$ and any $V$ on $\R^3$, one has
\begin{equation}
 \label{eq:mainpot3d}
\Tr\left( \hat H_B + V \right)_- \leq
\rho_3 \ \frac{B}{(2\pi)^2} \sum_{m=0}^\infty \int_{\R^3\times\R} \left((2m+1)B +p_3^2 + V(x) \right)_- \,dx\,dp_3
\end{equation}
with $\rho_3=\sqrt 3\, \pi$.
\end{corollary}

\begin{proof}
From the operator-valued Lieb-Thirring inequality of \cite{DLL} we know that
$$
\Tr\left( \hat H_B + V \right)_- 
\leq \frac\pi{\sqrt 3} \iint_{\R^2} \Tr_{L^2(\R^2)} \left( H_B + p_3^2 + V(\cdot,x_3) \right)_- \frac{dx_3\,dp_3}{2\pi} \,.
$$
Inequality \eqref{eq:mainpot3d} is therefore a consequence of Theorem \ref{mainpot}.
\end{proof}

For the harmonic oscillator we have

\begin{corollary}\label{main3d}
For any $B>0$, $\omega_1>0$, $\omega_2>0$, $\omega_3>0$ and $\mu>0$, inequality \eqref{eq:mainpot3d} holds with $\rho_3=1$ for $\hat V(x) = \omega_1^2x_1^2+\omega_2^2x_2^2 +\omega_3^2x_3^2-\mu$.
\end{corollary}

\begin{proof}
We denote by $E_j$ the eigenvalues of the one-dimensional harmonic oscillator $H= - \frac{\partial^2}{\partial x_3^2} + \omega_3^2x_3^2$. Then, since $\hat H_{B} + V = \left(H_{B} + V\right) \otimes I + I\otimes H$ with $V(x_1,x_2)=\omega_1^2x_1^2+\omega_2^2x_2^2$, we have
$$
\Tr_{L^2(\R^3)}\left( \hat H_{B} + \hat V \right)_- = \sum_j \Tr_{L^2(\R^2)}\left( H_{B} + V + E_j -\mu \right)_- \,.
$$
According to Theorem \ref{main} (which trivially holds for $\mu\leq 0$ as well), this is bounded from above by
\begin{align*}
& \frac{B}{2\pi} \sum_j \sum_{m=0}^\infty \int_{\R^2} \left((2m+1)B + V(x_1,x_2) +E_j-\mu \right)_- \,dx_1\,dx_2 \\
&\qquad = \frac{B}{2\pi} \sum_{m=0}^\infty \int_{\R^2} \Tr_{L^2(\R)} \left(H + (2m+1)B + V(x_1,x_2) -\mu \right)_- \,dx_1\,dx_2 \,.
\end{align*}
Next, we shall use that $H$ satisfies a Lieb-Thirring inequality with semi-classical constant \cite{dB,La2}, that is, for any $\Lambda\in\R$,
$$
\Tr_{L^2(\R)} \left(H - \Lambda \right)_- 
\leq \frac1{2\pi} \iint_{\R\times\R} \left(p_3^2 + \omega_3^2x_3^2 -\Lambda\right)_- \,dx_3\,dp_3 \,.
$$
(This can also be seen from Lemma \ref{mean} and recalling the explicit form of the eigenvalues of $H$.) It follows that for every fixed $(x_1,x_2)$
\begin{align*}
& \Tr_{L^2(\R)} \left(H + (2m+1)B + V(x_1,x_2) -\mu \right)_- \\
& \qquad \leq \frac1{2\pi} \iint_{\R\times\R} \left( p_3^2 + \omega_3^2x_3^2 + (2m+1)B+ V(x_1,x_2) -\mu \right)_- \,dx_3\,dp_3 \,,
\end{align*}
which proves the claimed bound.
\end{proof}

\begin{remark}
The previous proof shows that \eqref{eq:mainpot3d} with $\rho_3=1$ is valid for more general potentials $\hat V(x) = V(x_1,x_2) +v(x_3),$ where $V(x_1,x_2)=\omega_1^2x_1^2+\omega_2^2x_2^2$ and where $v$ is such that $\Tr_{L^2(\R)} (-\frac{d^2}{dx_3^2} +v(x_3)-\Lambda)_- \leq \frac1{2\pi} \iint_{\R\times\R} (p_3^2+v(x_3)-\Lambda)_- \,dx_3\,dp_3$ for all $\Lambda$.
\end{remark}

A similar argument as in the proofs of Corollaries \ref{mainpot3d} and \ref{main3d} (based on the operator-valued Lieb-Thirring inequalities of \cite{HLW,LaWe}) shows that for general $V$ one has
\begin{equation}
 \label{eq:mainpot3dal}
\Tr\left( \hat H_B + V \right)_-^\gamma \leq
\rho_{3,\gamma}\ \frac{B}{(2\pi)^2} \sum_{m=0}^\infty \int_{\R^3\times\R} \left((2m+1)B +p_3^2 + V(x) \right)_-^\gamma \,dx\,dp_3
\end{equation}
with $\rho_{3,\gamma} = 6$ if $\gamma\geq 1/2$, with $\rho_{3,\gamma} = \pi\sqrt 3$ if $\gamma\geq 1$ and with $\rho_{3,\gamma}=3$ if $\gamma\geq 3/2$. Moreover, in the special case of $\hat V(x) = \omega_1^2x_1^2+\omega_2^2x_2^2 +\omega_3^2x_3^2-\mu$, \eqref{eq:mainpot3dal} holds with $\rho_3=1$ for $\gamma\geq 1$ and with $\rho_{3,\gamma}=2 \left( \gamma/(\gamma+1)\right)^\gamma$ for $0\leq\gamma<1$. The latter follows from the fact \cite{FrLoWe} that
$$
\Tr_{L^2(\R)} \left( - \frac{d^2}{d x_3^2} + \omega_3^2x_3^2 - \Lambda \right)_-^\gamma 
\leq 2 \left(\frac \gamma{\gamma+1}\right)^\gamma \frac1{2\pi} \iint_{\R\times\R} \left(p_3^2 + \omega_3^2x_3^2 -\Lambda\right)_-^\gamma \,dx_3\,dp_3 \,.
$$


\section{Proof of Theorem \ref{mainpot}}

\subsection{A kinetic energy inequality}\label{sec:kinen}

We define a piecewise afine function $j:[0,\infty)\to[0,\infty)$ by
$$
j(\rho)=\frac{B^2}{2\pi} \left( L^2+ (2L+1) r\right)
\quad\text{if}\ \rho=\frac{B}{2\pi}(L+r),\ L\in\N_0,\ r\in[0,1) \,.
$$
We note that $j$ is continuous, increasing and convex. One has $j(\rho)=B\rho$ if $\rho\leq B/(2\pi)$ and $j(\rho)\sim 2\pi\rho^2$ if $\rho\gg B$. The connection between this function and the right side of \eqref{eq:mainpot} will become clearer in the next subsection.

\begin{theorem}\label{kinen}
 Let $0\leq\gamma\leq 1$ be a density matrix on $L^2(\R^2)$ with finite kinetic energy. Then
 $$
 \tr H_B \gamma \geq 3 \int_{\R^2} j(\rho_\gamma(x)/3) \,dx \,,
 $$
 where $\rho_\gamma(x)=\gamma(x,x)$.
\end{theorem}

It is easy to see that $3\, j(\rho/3)\geq (1/3)\, j(\rho)$ for all $\rho\geq 0$, and therefore we also have
$$
\tr H_B \gamma \geq (1/3) \int_{\R^2} j(\rho_\gamma(x)) \,dx \,.
$$

\begin{proof}
 The first part of our proof follows the method introduced by Rumin \cite{Ru}. We define $j_R:[0,\infty)\to[0,\infty)$ by
$$
j_R(\rho)=B\rho + 2B \sum_{k=1}^\infty \left(\sqrt\rho - \sqrt{\tfrac{Bk}{2\pi}}\right)_+^2 \,.
$$
We note that $j_R$ is differentiable and convex, $j_R(\rho)=B\rho$ if $\rho\leq B/(2\pi)$ and 
$j_R(\rho)\sim 2\pi\rho^2/3$ if $\rho\gg B$. We shall first show that
\begin{equation}\label{eq:rumin}
 \tr H_B \gamma \geq \int_{\R^2} j_R(\rho_\gamma(x)) \,dx \,.
\end{equation}
In the second part of our proof (see Lemma \ref{rumcomp}) we show that $j_R(\rho)\geq 3\ j(\rho/3)$ for all $\rho\geq 0.$

For the proof of \eqref{eq:rumin} we write
\begin{equation}\label{eq:layer}
\tr H_B \gamma = \int_0^\infty \tr\left( P^E \gamma\right) \,dE = \int_{\R^2} \int_0^\infty \rho^E_\gamma(x) \,dE\,dx \,,
\end{equation}
where $P^E$ is the spectral projection of $H_B$ corresponding to the interval $[E,\infty)$ and where $\rho^E_\gamma(x)=(P^E\gamma P^E) (x,x)$. It is well-known that
$$
(1-P^E)(x,x)=\frac{B}{2\pi}\#\{m\in\N_0:\ (2m+1)B < E\} \,.
$$
The same clever use of the triangle inequality as in \cite{Ru} leads to the pointwise lower bound
$$
 \rho^E_\gamma(x) \geq \left( \sqrt{\rho_\gamma(x)} - \sqrt{ \frac{B}{2\pi}\#\{m\in\N_0:\ (2m+1)B < E\} } \right)_+^2 \,.
$$
Inserting this bound in \eqref{eq:rumin} we obtain
\begin{align*}
\tr H_B \gamma 
& \geq \int_{\R^2} \left( \int_0^B \rho_\gamma(x) \,dE + \sum_{k=1}^\infty \int_{(2k-1)B}^{(2k+1)B}
\left( \sqrt{\rho_\gamma(x)} - \sqrt{ \frac{Bk}{2\pi} } \right)_+^2 \,dE \right) \,dx \\
& = \int_{\R^2} j_R(\rho_\gamma(x)) \,dx \,.
\end{align*}
This completes the proof of \eqref{eq:rumin} and also, by Lemma \ref{rumcomp} below, the proof of the theorem.
\end{proof}

\begin{lemma}\label{rumcomp}
 $j_R(\rho)\geq 3\ j(\rho/3)$ for all $\rho\geq 0$.
\end{lemma}

\begin{proof}
We are going to prove that
\begin{equation}\label{eq:rumcomp}
j_R(3\rho)\geq 3\ j(\rho) \,.
\end{equation}
Note that this is an equality for $\rho\leq B/(6\pi)$. Moreover, since the left side of \eqref{eq:rumcomp} is convex and the right side linear for $\rho\leq B/(2\pi)$, we conclude that \eqref{eq:rumcomp} holds for all $\rho\leq B/(2\pi)$.

Henceforth we shall assume that $\rho\geq B/(2\pi)$ and we write $3\rho=(B/2\pi)(K+s)$ with $K\in\N$ and $s\in[0,1)$. If $K=3L+m$ with $L\in\N$ and $m\in\{0,1,2\}$, then the lemma says that
$$
K+s + 2 \sum_{k=1}^K \left(\sqrt{K+s} - \sqrt{k}\right)^2 \geq 3 \left( L^2+ \tfrac13 (2L+1) (m+s) \right)\,.
$$
We expand the square on the left side and insert $L=(K-m)/3$ on the right side. This shows that the assertion is equivalent to
$$
K+s + 2 K(K+s) -4 \sqrt{K+s}\sum_{k=1}^K \sqrt{k} +K(K+1)
\geq \tfrac13 K^2 + \tfrac23 Ks +s + R \,,
$$
for $K\in\N$ and $s\in[0,1)$, where $R=-\frac13 m^2 -\frac23 ms+m$. Since the inequality has to be true for any $m\in\{0,1,2\}$, we can replace $R$ by its maximum over these $m$ (with fixed $s$), that is, by $(2/3)(1-s)$. Thus \eqref{eq:rumcomp} is equivalent to
$$
4K^2+(3+2s)K-6 \sqrt{K+s} \sum_{k=1}^K \sqrt{k} - 1+ s \geq 0 \,.
$$
The proof is straightforward for $K=1$ and we may therefore assume that $K\geq 2$. By the concavity of the square root we have 
$$\frac{\sqrt{k}+\sqrt{k+1}}{2}\leq \int_k^{k+1}\sqrt{t}\,dt \,.$$
Summing this from $k=1$ to $k=K-1$ we get
$$\sum_{k=1}^{K}\sqrt{k}\leq \int_1^K\sqrt{t}\,dt+\frac{1+\sqrt{K}}{2}
=\frac{2K^{3/2}}{3}+\frac{K^{1/2}}{2}-\frac{1}{6}\,.$$
This shows that
\begin{align*}
& 4K^2+(3+2s)K-6 \sqrt{K+s} \sum_{k=1}^K \sqrt{k} -1+s \\
& \quad \geq 4K^2+(3+2s)K- \sqrt{K(K+s)} (4K+3) +\sqrt{K+s} -1+s \\
& \quad = \frac{sK((4s-12)K-9)}{ 4K^2  + (3+2s)K +\sqrt{K(K+s)}(4K+3)}+\sqrt{K+s} -1+s \,.
\end{align*}
In the quotient on the right side we estimate the numerator from below by $-3sK(4K+3)$ and the denominator from below by $4K^2  + 3K +K(4K+3)= 2K(4K+3)$. Thus the quotient is bounded from below by $-3s/2$, and we conclude that
\begin{align*}
4K^2+(3+2s)K-6 \sqrt{K+s} \sum_{k=1}^K \sqrt{k} -\frac{3R}2 
\geq \sqrt{K+s} - 1 -\frac s2 \,.
\end{align*}
The right side is easily seen to be positive for $K\geq 2$ and $s\in[0,1)$, and this concludes the proof.
\end{proof}


\subsection{Proof of Theorem \ref{mainpot}}

In this section we are going to deduce Theorem \ref{mainpot} from Theorem \ref{kinen}. We define
$$
p(v) := -\frac{B}{2\pi} \sum_{m=0}^\infty \left((2m+1)B + v \right)_-
$$
for $v\in\R$. This is a convex, decreasing and non-positive function. The key observation is that this $p$ is the Legendre transform of the function $j$ from the previous subsection, that is,
\begin{equation}\label{eq:legendre}
p(v) = \inf_{\rho\geq 0} \left( j(\rho) + v\rho \right) \,.
\end{equation}
This can be verified by elementary computations.

In order to prove Theorem \ref{mainpot} we apply Theorem \ref{kinen} to get the estimate 
$$
\Tr\left( H_B + V \right)\gamma 
\geq \int_{\R^2} \left( 3 j(\rho_\gamma(x)/3) + V(x)\rho_\gamma(x) \right) \,dx
$$
for any $0\leq\gamma\leq 1.$ According to \eqref{eq:legendre} this is bounded from below by $3 \int_{\R^2} p(V(x)) \,dx$. For $\gamma$ equal to the projection corresponding to the negative spectrum of $H_B + V$ we obtain the assertion of Theorem \ref{mainpot}.

\begin{remark}
Similar arguments show that Theorem \ref{kinen} can be deduced from Theorem \ref{mainpot}. Indeed, since $j$ is convex it is its double Legendre transform. By \eqref{eq:legendre} we obtain
\begin{equation}\label{eq:legendre2}
j(\rho) = \inf_{v\in\R} \left( p(v) + v\rho \right) \,.
\end{equation}
By the variational principle and Theorem \ref{mainpot} we can estimate for any $0\leq \gamma\leq 1$ and any $V$
$$
\Tr H_B \gamma 
\geq - \Tr\left(H_B +V\right)_- + \int_{\R^2} V(x)\rho_\gamma(x) \,dx
\geq \int_{\R^2} \left( 3 p(V(x)) + V(x)\rho_\gamma(x) \right) \,dx \,.
$$
According to \eqref{eq:legendre2} this is bounded from below by $3 \int_{\R^2} j(\rho_\gamma(x)/3) \,dx$, and this shows Theorem~\ref{kinen}.
\end{remark}



\section{Proof of Theorem \ref{main}}

\subsection{The spectrum of $H_{B} + V$}

The explicit form of the eigenvalues of $H_{B} + \omega^2|x|^2$ was discoverd in \cite{Fo}. We include an alternative derivation of this result, which is also valid in the non-radial case.

\begin{lemma}\label{explicit}
 For any $B>0$ and $\omega_1,\omega_2>0$ the operator $H_{B} + \omega_1^2x_1^2+\omega_2^2x_2^2$ has discrete spectrum and its eigenvalues, including multiplicities, are given by
$$
B \left( a_+(\tfrac{\omega_1}B,\tfrac{\omega_2}B) \, (2k+1) + a_-(\tfrac{\omega_1}B,\tfrac{\omega_2}B) \, (2l+1) \right) \,,
\qquad k,l\in\N_0 \,,
$$
where
\begin{equation}
 \label{eq:lambda}
a_\pm(\sigma_1,\sigma_2) = \sqrt{ \tfrac12 \left( 1+\sigma_1^2+\sigma_2^2 \pm \sqrt{(1+\sigma_1^2+\sigma_2^2)^2 -4\sigma_1^2\sigma_2^2} \right) } \,.
\end{equation}
\end{lemma}

\begin{remark}
It will be important for our analysis below that
\begin{equation}
 \label{eq:lambdadiff}
a_-(\sigma)\, a_+(\sigma) = \sigma_1\sigma_2 \,,
\end{equation}
which is easily checked.
\end{remark}

\begin{proof}
 By means of the gauge transform $e^{-iBx_1x_2/2}$ we see that $H_{B} + V$ is unitarily equivalent to the operator
$$
- \frac{\partial^2}{\partial x_1^2} + \left(-i \frac\partial{\partial x_2} - Bx_1 \right)^2 + \omega_1^2 x_1^2 + \omega_2^2 x_2^2 \,,
$$
which, in turn, by a partial Fourier transform with respect to $x_2$, is unitarily equivalent to
$$
- \frac{\partial^2}{\partial x_1^2} + \left( x_2 - Bx_1 \right)^2 + \omega_1^2 x_1^2 -\omega_2^2 \frac{\partial^2}{\partial x_2^2} \,.
$$
After scaling $x_2\mapsto \omega_2 x_2$ this becomes the non-radial harmonic oscillator $-\Delta + x^t A x$ with the matrix
$$
A = \begin{pmatrix} B^2+\omega_1^2 & -B\omega_2\\ -B\omega_2& \omega_2^2 \end{pmatrix} \,.
$$
The eigenvalues of $A$ are $B^2a_+(\omega_1/B,\omega_2/B)^2$ and $B^2a_-(\omega_1/B,\omega_2/B)^2$. Using the eigenvectors of $A$ as basis in $\R^2$, we obtain a direct sum of two one-dimensional harmonic oscillators with frequencies $Ba_+$ and $Ba_-$, respectively. This leads to the stated form of the eigenvalues.
\end{proof}

According to Lemma \ref{explicit} and a simple computation, \eqref{eq:mainpot} with $\rho_2=1$ is equivalent to
$$
\sum_{k,l\ge 0} \!\!\left( \mu - B a_+(\tfrac{\omega_1}B,\tfrac{\omega_2}B) (2k+1) - B a_-(\tfrac{\omega_1}B,\tfrac{\omega_2}B)  (2l+1) \right)_+
\leq \frac{B}{4\omega_1\omega_2} \sum_{m\ge 0} \!\left(\mu- (2m+1)B\right)_+^2
$$
with $a_\pm$ given by \eqref{eq:lambda}. Setting $\Lambda=\mu/B$, $\sigma_j=\omega_j/B$ and $a_\pm=a_\pm(\sigma)$ and substituting \eqref{eq:lambdadiff} we can rewrite the desired inequality as
\begin{equation}
\label{eq:mainequiv}
\sum_{k,l\ge 0} \left( \Lambda - a_+\, (2k+1) - a_-\, (2l+1) \right)_+
\leq \frac{1}{4 a_- a_+} \sum_{m\ge0} \left(\Lambda- (2m+1)\right)_+^2 \,,
\end{equation}
and this is what we shall prove.


\subsection{Two inequalities for convex functions}

For the proof of \eqref{eq:mainequiv} we shall need

\begin{lemma}\label{mean}
Let $\phi$ be a non-negative convex function on $(0,\infty)$ such that $\int_0^\infty \phi(t) \,dt$ exists. Then
$$
\sum_{k=0}^\infty \phi(k+\tfrac12) \leq \int_0^\infty \phi(t) \,dt \,.
$$
\end{lemma}

\begin{proof}
Indeed, by the mean-value property of convex functions $\phi(k+\tfrac12) \leq \int_{k}^{k+1} \phi(t) \,dt$ for each $k$. Now sum over $k$.
\end{proof}

\begin{remark}\label{mean1}
 The proof also shows that $\sum_{k=0}^{K-1} \phi(k+\tfrac12) \leq \int_0^{K} \phi(t) \,dt$ for each integer $K$. This observation will be useful later.
\end{remark}

The inequality from Lemma \ref{mean} is sufficient to prove a sharp Lieb-Thirring inequality in the non-magnetic case, but for the proof of our Theorem \ref{main} we need a more subtle fact about convex functions. We note that by the previous lemma $h \sum_{k=0}^\infty \phi(h(k+\tfrac12)) \leq \int_0^\infty \phi(t) \,dt$ for any $h>0$. Moreover, $h \sum_{k=0}^\infty \phi(h(k+\tfrac12)) \to \int_0^\infty \phi(t) \,dt$ as $h\to 0$ by the definition of the Riemann integral. The key for proving our sharp result is that, for a certain subclass of convex functions, this limit is approached monotonically. More precisely, one has

\begin{lemma}\label{mono}
Let $\phi$ be a non-negative convex function on $(0,\infty)$ such that $\int_0^\infty \phi(t) \,dt$ exists. Assume that $\phi$ is differentiable and that $\phi'$ is concave. Then the sum
$$
h \sum_{k=0}^\infty \phi(h(k+\tfrac12))
$$
is decreasing in the parameter $h>0$.
\end{lemma}

We emphasize that without assumptions on $\phi'$ the inequality
$$
\sum_{k=0}^\infty \phi(k+\tfrac12) \leq
h \sum_{k=0}^\infty \phi(h(k+\tfrac12)) 
$$
is not true for all $h< 1$. Indeed, take for instance $\phi(t)=(1-t)_+$ and $h\geq 2/3$.

In the proof of this lemma we shall make use of the following well-known fact about convex functions: If $\psi$ is a non-negative convex function on $(0,\infty)$ such that $\int_0^\infty \psi(t) \,dt$ exists, then $\psi(t)=\int_0^\infty(T-t)_+ \,d\mu(T)$ for some non-negative measure $\mu$. Indeed, it is known that the left-sided derivative $\partial_-\psi$ exists everywhere on $(0,\infty)$ and satisfies $\psi(b)-\psi(a)=\int_a^b \partial_-\psi(t)\,dt$ for $0<a<b<\infty$. Moreover, $\partial_-\psi$ is increasing and left-continuous, and therefore there is a non-negative measure $\mu$ such that $\partial_-\psi(b)-\partial_-\psi(a)=\mu([a,b))$. Since $\lim_{t\to\infty}\psi(t)=\lim_{t\to\infty}\partial_-\psi(t)=0$, we have by Fubini's theorem
$$
\psi(t) = -\int_t^\infty \partial_-\psi(a)\,da = \int_t^\infty \left( \int \chi_{[a,\infty)}(T) \,d\mu(T) \right) \,da
= \int_0^\infty(T-t)_+ \,d\mu(T) \,,
$$
as claimed.

\begin{proof}
By the fact recalled above (with $\psi=-\phi'$) we have $\phi(t)=\int_0^\infty(T-t)_+^2 \,d\mu(T)$ for a non-negative measure $\mu$. Hence it suffices to prove the lemma for $\phi(t)=(T-t)_+^2$ with $T>0$. We have to prove that $\sum_{k=0}^\infty \left( \phi(h(k+\tfrac12)) + h(k+\tfrac12)\phi'(h(k+\tfrac12)\right)\leq 0$, which for our $\phi$ reads
$$
\sum_{k=0}^\infty \left( (S - 2k-1)_+^2 - 2(2k+1)(S - 2k-1)_+ \right)\leq 0 \,
$$
with $S=2T/h$. Choose $K\in\N_0$ such that $2K+1\leq S < 2K+3$. Then the left side above equals
\begin{align*}
& \sum_{l=0}^K \left( (S - 2k-1)^2 - 2(2k+1)(S - 2k-1) \right)
= \sum_{l=0}^K \left( S^2 -4S (2k+1) +3(2k+1)^2 \right) \\
& \qquad = (K+1) \left( S^2 -4S (K+1) + (2K+1) (2K+3) \right) \\
& \qquad = (K+1) (S - 2K-1)(S - 2K-3) \,.
\end{align*}
This is clearly non-positive for $2K+1\leq S < 2K+3$, thus proving the claim.
\end{proof}


\subsection{Proof of Theorem \ref{main}}

We have to prove \eqref{eq:mainequiv}. By Lemma \ref{mean} for any $k$
\begin{align*}
\sum_{l\geq 0} \left(\Lambda - a_+(2k+1) -a_-(2l+1)\right)_+ 
& \leq \int_0^\infty \left(\Lambda - a_+(2k+1) -2a_- t\right)_+ \,dt \\
& = \frac{1}{4a_-} \left(\Lambda - a_+(2k+1) \right)_+^2 \,.
\end{align*}
A simple computation shows that $a_+=a_+(\sigma)\geq 1$, and hence by Lemma \ref{mono}
$$
a_+ \sum_{k\geq 0} \left(\Lambda - a_+(2k+1) \right)_+^2
\leq \sum_{k\geq 0} \left(\Lambda - (2k+1) \right)_+^2 \,.
$$
The previous two inequalities imply the desired \eqref{eq:mainequiv}.
\qed


\subsection{Proof of Proposition \ref{sharpness}}\label{sec:gamma-}

Given $0\leq\gamma<1$, we want to find $\omega_1=\omega_2$ and $B$ such that the reverse inequality \eqref{eq:gamma-} holds. We may assume $\gamma>0$ in the following. (The case $\gamma=0$ can be treated similarly, or one may use the argument of Aizenman and Lieb mentioned in the introduction to conclude that a counterexample for $\gamma=\gamma_0$ implies one for all $\gamma<\gamma_0$.)

By the same computation that lead to \eqref{eq:mainequiv} we see that \eqref{eq:gamma-} can be written as
\begin{align*}
 \sum_{k,l\ge 0} \left( \Lambda - a_+\, (2k+1) - a_-\, (2l+1) \right)_+^{\gamma}
> \frac{1}{2(\gamma+1)a_- a_+} \sum_{m\ge0} \left(\Lambda- (2m+1)\right)_+^{\gamma+1}
\end{align*}
with $\Lambda=\mu/B$, $\sigma_j=\omega_j/B$ and $a_\pm=a_\pm(\sigma)$. We will let $\omega_1=\omega_2$ and use the notation $t=\sigma^2$. One can show that 
$a_+=1+t+O(t^2)$
and 
$a_-=t+O(t^2)$
as $t\to 0+.$
We now choose $\Lambda=3$ and recall that $a_+=a_+(\sigma)\geq 1$. This gives us the inequality
$$2(\gamma+1)a_- a_+ \sum_{l\ge 0} \left( 3 - a_+\, - a_-\, (2l+1) \right)_+^{\gamma}-2^{\gamma+1}>0 \,,$$
which may be written as
$$(\gamma+1)a_-^{\gamma+1}a_+\sum_{l\ge 0}(x-l)_+^\gamma - 1>0$$
with $x=(3-a_+-a_-)/(2a_-)$. Since $x=t^{-1}(1+O(t))$ as $t\to 0+$, we may choose $\sigma$ so that $x$ is an integer. In this case we may use the concavity of $y^\gamma$ and Remark \ref{mean1} to bound
$$\sum_{l\ge 0}(x-l)_+^\gamma=\sum_{l=1}^x l^\gamma \ge \int_{1/2}^{x+1/2}t^{\gamma}\,dt=\frac{1}{\gamma+1}((x+1/2)^{\gamma+1}-(1/2)^{\gamma+1}) \,.$$
This shows that 
\begin{align*}
(\gamma+1)a_-^{\gamma+1} a_+ \sum_{l\ge 0}(x-l)_+^\gamma 
&\ge a_+ \left((a_-x+a_-/2)^{\gamma+1}-(a_-/2)^{\gamma+1}\right)\\
&=a_+ \left(((3-a_+)/2)^{\gamma+1}-(a_-/2)^{\gamma+1}\right)\\
&=\left(1+t+O(t^2)\right)\left(1-t/2+O(t^2)\right)^{\gamma+1}+O(t^{\gamma+1})\\
&=1+\frac{1-\gamma}{2}t+O(t^{\gamma+1}) \,.
\end{align*}
Since this is strictly larger than $1$ for sufficiently small $t$, we have proved our claim.
\qed


\bibliographystyle{amsalpha}

\end{document}